\documentclass[12pt, reqno]{amsart}
\usepackage{amsmath, amsthm, amscd, amsfonts, amssymb, graphicx, color}
\usepackage[bookmarksnumbered, colorlinks, plainpages]{hyperref}
\hypersetup{colorlinks=true,linkcolor=red, anchorcolor=green, citecolor=cyan, urlcolor=red, filecolor=magenta, pdftoolbar=true}

\newtheorem{theorem}{Theorem}[section]
\newtheorem{lemma}[theorem]{Lemma}
\newtheorem{proposition}[theorem]{Proposition}
\newtheorem{corollary}[theorem]{Corollary}
\theoremstyle{definition}

\theoremstyle{remark}
\newtheorem{remark}[theorem]{Remark}
\numberwithin{equation}{section}

\begin{document}
\setcounter{page}{1}

\title{Reverses of operator Acz\'{e}l inequality}

\author[ V. Kaleibary and S. Furuichi]{Venus Kaleibary and Shigeru Furuichi}

\address{ School of Mathematics, Iran University of Science and Technology, Narmak, Tehran 16846-13114, Iran.}
\email{\textcolor[rgb]{0.00,0.00,0.84}{v.kaleibary@gmail.com}}

\address{ Department of Information Science, College of Humanities and Sciences, Nihon University,
3-25-40, Sakurajyousui, Setagaya-ku, Tokyo, 156-8550, Japan.}
%\newline
%Tusi Mathematical Research Group (TMRG), Mashhad, Iran.}
\email{\textcolor[rgb]{0.00,0.00,0.84}{furuichi@chs.nihon-u.ac.jp}}

%\dedicatory{This paper is dedicated to Professor ABCD}



%\date{Received: xxxxxx; Revised: yyyyyy; Accepted: zzzzzz.}
%\newline \indent $^{*}$ Corresponding author}

\begin{abstract}
In this paper we present some inequalities involving operator decreasing functions and operator means. These inequalities provide some reverses of operator Acz\'{e}l inequality dealing with the weighted geometric mean.
\end{abstract} \maketitle

\section{Introduction}

Let $ B(\mathcal{H})$ denote the $C^*$-algebra of all bounded linear operators on a Hilbert space $( \mathcal{H}, \langle \cdot , \cdot \rangle)$. An operator $A \in B(\mathcal{H})$ is called \textit{positive} if $\langle Ax , x \rangle \geq 0$  for every $x \in \mathcal{H}$ and then we write $A \geq 0$. For self-adjoint operators $A,B \in  B(\mathcal{H})$, we say $A \leq B$ if $B - A \geq 0$. Also we say $A$ is \textit{positive definite} and we write $A > 0$, if $\langle Ax , x \rangle > 0$  for every $x \in \mathcal{H}$. Let $f$ be a continuous real function on $(0,\infty)$. Then $f$ is said to be \textit{operator monotone}
(more precisely, operator monotone increasing) if $A \geq B$ implies $f(A) \geq f(B)$ for positive definite operators $A, B $, and \textit{operator monotone decreasing} if $-f$ is operator monotone or
$A \geq B$ implies $f(A) \leq f(B)$.\\
Also, $f$ is said to be \textit{operator convex} if $f(\alpha A+(1- \alpha ) B) \leq \alpha f(A)+(1- \alpha)f(B)$
for all positive definite operators $A,B$ and $ \alpha \in [0, 1]$, and \textit{operator concave} if $-f$ is operator convex.

In 1956, Acz\'{e}l \cite{Ac} proved that if $a_i, b_i (1 \leq i \leq n)$ are positive real numbers such that
$a_1^2 - \sum_{i=2}^n a_i^2 >0 $ and $b_1^2 - \sum_{i=2}^n b_i^2 >0 $, then
\begin{align*}
 \big( a_1^2 - \sum_{i=2}^n a_i^2 \big) \big( b_1^2 - \sum_{i=2}^n b_i^2 \big) \leq \big( a_1 b_1 - \sum_{i=2}^n a_i b_i \big)^2.
\end{align*}
As is well known, Acz\'{e}l's inequality has important applications in the theory of functional equations in
non-Euclidean geometry. In recent years, considerable attention has been given to this inequality involving its
generalizations, variations and applications. See \cite{Ch, D1, P} and references therein. 
Popoviciu \cite{P} first presented an exponential extension of Acz\'{e}l's
inequality which states if $p>0, q>0 , \frac{1}{p}+\frac{1}{q}=1$, 
$a_1^p - \sum_{i=2}^n a_i^P >0 $, and $b_1^q - \sum_{i=2}^n b_i^q >0 $, then
\begin{align*}
 \big( a_1^P - \sum_{i=2}^n a_i^P \big)^\frac{1}{p} \big( b_1^q - \sum_{i=2}^n b_i^q \big)^\frac{1}{q} \leq  a_1 b_1 - \sum_{i=2}^n a_i b_i.
\end{align*}
 A variant of Acz\'{e}l's inequality in inner product spaces was given by
Dragomir \cite{D1} by establishing that if $a, b$ are real numbers and $x, y$ are vectors of an inner product space
such that $a^2- \| x \|^2 > 0$ or $b^2-\| y \|^2 > 0$, then 
\begin{align}\label{e14}
(a^2- \| x \|^2)(b^2-\| y \|^2) \leq (ab-Re \langle x, y \rangle)^2.
\end{align}
Moslehian in \cite{M} proved an operator version of the classical Acz\'{e}l inequality involving $\alpha$-geometric mean 
$A \sharp_\alpha B = A^{1/2} ( A^{-1/2} B A^{-1/2})^\alpha A^{1/2}$. He proved that 
if $ g $ is a non-negative operator decreasing and operator concave function on $(0, \infty)$ , $\frac{1}{p}+ \frac{1}{q}=1, p,q>1 $, and $A$ and $B$ are positive definite operators, then 
\begin{align} 
&g(A^p)\sharp_\frac{1}{q} g(B^q) \leq g (A^p \sharp_\frac{1}{q} B^q), \label{e9}\\ 
& \langle g(A^p) \xi , \xi \rangle ^\frac{1}{p} \langle g(B^q) \xi , \xi \rangle^\frac{1}{q}\leq  \langle g(A^p \sharp_\frac{1}{q}B^q) \xi , \xi \rangle \label{e10}
\end{align}
for all $\xi \in \mathcal{H}$.

In this paper we present some reverses of operator Acz\'{e}l inequlilities \eqref{e9} and \eqref{e10}, by using several reverse Young's inequalities. In fact, we establish some upper bounds for inequlilities \eqref{e9} and \eqref{e10}. These results are proved for a non-negative operator decreasing function $g$ and the condition of operator concavity has been omitted. So, we use less restrective conditions on $g$. The statements are organized in two sections respect to different coefficients.

 \section{Reverse inequalities via Kantorovich constant }
 
Let $A$ and $B$, be positive definite operators. For each $\alpha \in [0, 1]$ the \textit{$\alpha$-arithmetic} mean is defined as
$A\bigtriangledown_\alpha B := (1-\alpha)A + \alpha B$ and the \textit{$\alpha$-geometric} mean is 
\begin{equation*}
A \sharp_\alpha B = A^{1/2} ( A^{-1/2} B A^{-1/2})^\alpha A^{1/2}.
\end{equation*}
Clearly if $AB = BA$, then $A\sharp_\alpha B = A^{1-\alpha}B^\alpha$.
Basic properties of the arithmetic and geometric means can be found in \cite{F}.
It is well-known the Young inequality
\begin{align*}
A \sharp_\alpha B \leq (1- \alpha) A + \alpha B.
\end{align*}

 The celebrated Kantorovich constant is defined by
\begin{align}\label{e12}
K(t) = \dfrac{(t+1)^2}{4t}, \hspace*{1cm} t>0. 
\end{align}
The function $K$ is decreasing on $(0, 1)$ and increasing on $[1, \infty)$, $K(t)=K(\frac{1}{t})$, and $K(t)\geq 1$ for every $t>0$ \cite{F}.
 
The research on the Young inequality is interesting and there are several multiplicative and additive reverses of this inequality \cite{D, LW}. One of this reverse inequalities, is given by Liao et al. \cite{LW} using the Kantorovich constant as follows:
\begin{lemma}\label{l4}
\cite[Theorem 3.1]{LW}
Let $A, B$ be positive operators satisfying the following conditions $ 0 < m I \leq A \leq m^{'} I \leq M^{'} I \leq B \leq M I$ or $ 0 < m I \leq B \leq m^{'} I \leq M ^{'} I \leq A \leq M I$, for some constants $m, m^{'}, M, M^{'}$. Then
\begin{align} \label{e}
(1-\alpha ) A+ \alpha B \leq K(h)^R ( A \sharp_\alpha B ),
\end{align}
where $h = \frac{M}{m}$, $\alpha \in [0 , 1]$, $R = \max \lbrace 1-\alpha, \alpha\rbrace$ and $K(h)$ is the Kantorovich constant defined as \eqref{e12}.
\end{lemma}

In the following, we generalize Lemma \ref{l4} with the more general sandwich condition $ 0 < sA \leq B \leq tA$. The sketch of proof is similar to that of \cite[Theorem 2.1]{T}.
\begin{lemma} \label{l1}
 Let $ 0 < s A \leq  B \leq t A$, for some scalars $0 < s \leq t $ and $\alpha \in [0, 1]$. Then 
\begin{align} \label{e1}
(1-\alpha ) A+ \alpha B \leq  \max \lbrace  K(s)^R, K(t)^R \rbrace ( A \sharp_\alpha B ),
\end{align}
where $R= \max \lbrace \alpha, 1-\alpha \rbrace $ and $K(t)$ is the Kantorovich constant defined as \eqref{e12}.
\end{lemma}
\begin{proof}
From \cite[Corollary 2.2]{LW} if $x$ is a positive number and $\alpha \in [0, 1]$, then 
\begin{align*} \label{}
(1-\alpha) + \alpha x \leq K(x)^R  x^{\alpha}.
\end{align*}
 Thus for the positive definite operator $0 < s I \leq C \leq tI$, we have
\begin{align*} \label{}
(1-\alpha) + \alpha C \leq \max_{s \leq x \leq t} K(x)^R C^{\alpha}. 
\end{align*}
Substituting $A^{-\frac{1}{2}} B A^{-\frac{1}{2}}$ for $C$, we get
\begin{align*} \label{}
(1-\alpha) + \alpha A^{-\frac{1}{2}} B A^{-\frac{1}{2}} \leq \max_{s \leq x \leq t}  K(x)^R (A^{-\frac{1}{2}} B A^{-\frac{1}{2}})^{\alpha}. 
\end{align*}
Multiplying $A^{-\frac{1}{2}}$ to the both sides in the above inequality, and using the fact that $\max_{s \leq x \leq t} K(x) = max \lbrace K(s), K(t) \rbrace $, the desired inequality is obtained.

\end{proof}
\begin{remark}
We remark that Lemma \ref{l1} is a generalization of Lemma \ref{l4}. Since, if $ 0 < m I \leq A \leq m^{'} I \leq M^{'} I \leq B \leq M I$ or $ 0 < m I \leq B \leq m^{'} I \leq M ^{'} I \leq A \leq M I$, then $\frac{m}{M} A \leq B \leq \frac{M}{m} A$. Now by letting $s = \frac{m}{M}$ and $t = \frac{M}{m}$ in Lemma \ref{l1}, inequality \eqref{e} is obtained.  Note that $K(t) = K(\frac{1}{t})$,  for every $t>0$.
\end{remark}

\begin{lemma}\label{l2}
Let $ g $ be a non-negative operator monotone decreasing function on $(0, \infty)$ and $A$ be a positive definit operator. Then, for every scalar $\lambda \geq 1$
\begin{align*}
 \frac{1}{\lambda}g(A) \leq g (\lambda A ).
\end{align*}
\end{lemma}
\begin{proof}
First note that since $g$ is analytic on $(0 , \infty)$, we may assume that $g(x) > 0$ for all $x > 0$; otherwise $g$ is identically zero. Also, since $g$ is an operator monotone decreasing on $(0, \infty )$, so $f= 1/g$ is operator monotone on $(0, \infty)$ and hence operator concave function \cite{B1}. On the other hand, it is known that for every  non-negative concave function $f$ and $\lambda \geq 1$, $ f (\lambda x) \leq \lambda f(x)$. Therefore, for every $\lambda \geq 1$ we have 
\begin{align*}
(g (\lambda A))^{ -1 } \leq \lambda (g(A))^{-1}.
\end{align*}
Reversing this inequality, gives the result.
\end{proof} 

\begin{proposition}\label{p1}
 Let $ g $ be a non-negative operator monotone decreasing function on $(0, \infty)$ and $0 < s A \leq  B  \leq t A$ for some constants $0 < s \leq t$. Then, for all $\alpha \in [0 , 1]$
\begin{align} \label{e18}
 g( A \sharp_\alpha B) \leq  \max \lbrace  K(s)^R, K(t)^R \rbrace  (g(A)\sharp_\alpha g(B)),
\end{align}
where $R= \max \lbrace \alpha, 1-\alpha \rbrace $ and $K(t)$ is the Kantorovich constant defined as \eqref{e12}.
\end{proposition}
\begin{proof}
Since $0 < s A \leq  B  \leq t A$, from Lemma \ref{l1} we have 
 \begin{align*} 
 (1-\alpha ) A+ \alpha B \leq  \lambda ( A \sharp_\alpha B ), 
\end{align*}
where $\lambda = \max \lbrace  K(s)^R, K(t)^R \rbrace$.  We know that $\lambda \geq 1$. Also, the function $g$ is operator monotone decreasing and so
 \begin{align*} 
  g ( \lambda ( A \sharp_\alpha B )) \leq  g((1-\alpha ) A+ \alpha B).
\end{align*}
Now we can write  
\begin{align*} \label{}
\frac{1}{\lambda}g (A \sharp_\alpha B)  \leq g(\lambda (A\sharp_\alpha B)) \leq g((1-\alpha) A + \alpha B) \leq g(A)\sharp_\alpha g(B),
\end{align*}
where the first inequality follows from Lemma \ref{l2} and the last inequality follows from \cite[Theorem 2.1]{Tf}. 
\end{proof}
\begin{theorem}\label{t1}
 Let $ g $ be a non-negative operator monotone decreasing function on $(0, \infty)$ , $\frac{1}{p}+ \frac{1}{q}=1, p,q>1$,  and $0 < s A^p \leq B^q  \leq t A^p$ for some constants $s, t$. Then, for all $\xi \in \mathcal{H}$
\begin{align} 
&g (A^p \sharp_\frac{1}{q} B^q)  \leq \max \lbrace  K(s)^R, K(t)^R \rbrace g(A^p)\sharp_\frac{1}{q} g(B^q), \label{e3}\\ 
& \langle g(A^p \sharp_\frac{1}{q}B^q) \xi , \xi \rangle \leq \max \lbrace  K(s)^R, K(t)^R \rbrace \langle g(A^p) \xi , \xi \rangle ^\frac{1}{p} \langle g(B^q) \xi , \xi \rangle^\frac{1}{q}, \label{e4}
\end{align}
where $R= \max \lbrace \frac{1}{p}, \frac{1}{q} \rbrace $, and $K(t)$ is the Kantorovich constant defined as \eqref{e12}.
\end{theorem}
\begin{proof}
Letting $\alpha = \frac{1}{q}$ and replacing $A^p$ and $B^q$ with $A$ and $B$ in Proposition \ref{p1}, we reach the inequlity \eqref{e3}. To prove inequality \eqref{e4}, first note that under the condition $0 < s A^p \leq B^q  \leq t A^p$ from Lemma \ref{l1} we have 
 \begin{align*}
A^p \nabla_\alpha B^q \leq  \max \lbrace  K(s)^R, K(t)^R \rbrace  ( A^p \sharp_\alpha B^q ).
\end{align*}
For conviniance set $M=\lbrace  K(s)^R, K(t)^R \rbrace $. So, for the operator monotone decreasing functions $g$ and $\alpha = \frac{1}{q}$
 \begin{align} \label{e5}
g ( M (A^p \sharp_\frac{1}{q}B^q)) \leq g(A^p \nabla_\frac{1}{q}B^q).
\end{align}
Now compute
\begin{align*} 
\langle g(A^p \sharp_\frac{1}{q}B^q) \xi , \xi \rangle
& \leq M \langle g ( M (A^p \sharp_\frac{1}{q}B^q)) \xi , \xi \rangle \nonumber\\
& \leq M \langle g ( A^p \nabla_\frac{1}{q}B^q)) \xi , \xi \rangle \nonumber\\ 
& \leq M \langle g (A^p) \sharp_\frac{1}{q} g (B^q) \xi , \xi \rangle \nonumber\\
& \leq M \langle g (A^p) \xi , \xi \rangle^{\frac{1}{p}} \langle g (B^q) \xi , \xi \rangle ^{\frac{1}{q}},\nonumber
\end{align*}
where the first inequality follows from Lemma \ref{l2} and  the second follows from the inequality \eqref{e5}. For the third inequality we use log-convexity property of operator monotone decreasing functions \cite[Theorem 2.1]{Tf}, and in the last inequality we use the fact that for every positive operators $A, B$ and every $\xi \in \mathcal{H}$,  $\langle A\sharp_\alpha B \xi , \xi \rangle \leq \langle A \xi , \xi \rangle ^{1-\alpha} \langle B \xi , \xi \rangle^{\alpha}$ \cite[Lemma 8]{BL}. So, we achieve
\begin{align*} 
\langle g(A^p \sharp_\frac{1}{q}B^q) \xi , \xi \rangle \leq \max \lbrace  K(s)^R, K(t)^R \rbrace \langle g(A^p) \xi , \xi \rangle ^\frac{1}{p} \langle g(B^q) \xi , \xi \rangle^\frac{1}{q}, 
\end{align*}
as desired.
\end{proof}
\begin{corollary}\label{c1}
Let $\frac{1}{p}+\frac{1}{q}=1$, $p, q\geq 1$, and $A, B$ be commuting positive operators with spectra contained in $(0, 1)$ such that $0 < s A^p \leq  B^q  \leq t A^p$ for some constants $s, t$. Then, for every unit vector $\xi \in \mathcal{H}$
\begin{align}\label{e15}
1- \| (A B)^\frac{1}{2} \xi \|^2 \leq \max \lbrace  K(s)^R, K(t)^R \rbrace (1-\| A^\frac{p}{2} \xi \|^2)^\frac{1}{p} (1-\| B^\frac{q}{2} \xi \|^2)^\frac{1}{q},
\end{align}
and consequently
\begin{align*}
1- \| A B \xi \|^2 \leq \max \lbrace  K(s^2)^R, K(t^2)^R \rbrace (1-\| A^p \xi \|^2)^\frac{1}{p} (1-\| B^q \xi \|^2)^\frac{1}{q},
\end{align*}
where $R = \max \lbrace \frac{1}{p}, \frac{1}{q} \rbrace$.
\end{corollary}
\begin{proof}
The first inequality is obtained by applying Theorem \ref{t1} to the function $g(t) = 1 - t$ on $(0, 1)$ and the fact that $A^p \sharp_\frac{1}{q} B^q = AB$. Also, for every positive operator $A$
\begin{align*}
\langle A \xi, \xi \rangle = \langle A^\frac{1}{2} \xi,   A^\frac{1}{2}\xi \rangle = \| A^\frac{1}{2} \xi \|.
\end{align*}
 For the second inequality note that since $AB = BA$, from the sandwich condition $0 < s A^p \leq  B^q  \leq t A^p$ we have $0 < s^2 A^{2p} \leq  B^{2q}  \leq t^2 A^{2p}$. Now replacing $A^2$ and $B^2$ with $A$ and $B$ in \eqref{e15}, the assertion is obtained.
\end{proof}
\begin{remark}
Moslehian in \cite[Corollary 2.4]{M}, showed operator version of Acz\'{e}l inequlity \eqref{e14} as follows:
\begin{align}\label{e13}
 (1-\| A^\frac{p}{2} \xi \|^2)^\frac{1}{p} (1-\| B^\frac{q}{2} \xi \|^2)^\frac{1}{q}  \leq 1- \| (A B)^\frac{1}{2} \xi \|^2, 
\end{align}
where $A$ and $B$ are commuting positive operators with spectra contained in $(0, 1)$, and $\frac{1}{p}+\frac{1}{q}=1$ for $p, q\geq 1$. As it seen, inequality \eqref{e15} in Corollary \ref{c1}, provides an upper bound for the operator Acz\'{e}l inequality \eqref{e13}. 
\end{remark}

\begin{corollary}
 Let $ g $ be a non-negative operator monotone decreasing function on $(0, \infty)$ and $A, B$ be commuting positive operators such that $0 < s A^p \leq  B^q  \leq t A^p$ for some constants $s, t$. Then
\begin{align*}
g(AB) \leq \max \lbrace  K(s)^R, K(t)^R \rbrace g(A^P)^\frac{1}{p} g(B^q)^\frac{1}{q},
\end{align*}
where $R= \max \lbrace \frac{1}{p}, \frac{1}{q} \rbrace $.
\end{corollary}
\begin{corollary}
 Let $ g $ be a non-negative decreasing function on $(0, \infty)$ and $a_i, b_i$ be positive numbers such taht $ 0 < s \leq \frac{(b_i)^q}{(a_i)^p } \leq t $ for some constants $s, t$. Then
\begin{align}\label{e16}
\sum_{i=1}^n g(a_i b_i) \leq \max \lbrace  K(s)^R, K(t)^R \rbrace \big( \sum_{i=1}^n g(a_i ^p) \big)^\frac{1}{p}  \big( \sum_{i=1}^n g(b_i ^q) \big)^\frac{1}{q}, 
\end{align}
where $R= \max \lbrace \frac{1}{p}, \frac{1}{q} \rbrace$.
\end{corollary}
\begin{proof}
Let $A(x_1, x_2, \cdots , x_n) = (a_1 x_1, a_2 x_2, \cdots, a_n x_n)$ and $B(x_1, x_2, \ldots , x_n) =$\\
$ (b_1 x_1, b_2 x_2, \cdots, b_n x_n)$ be positive operators acting on Hilbert space $\mathcal{H}= C^n$ and $\xi = (1,1,\ldots,1)$. Now by applying inequality \eqref{e4} to the operators $A$ and $B$, we get the inequality \eqref{e16}.
\end{proof}

\section{Some related results}
Dragomir in \cite[Theorem 6]{D}, gave another reverse inequality for Young's
inequality as follows :
\begin{lemma}\label{l3}
Let $A, B$ be positive operators such that $0 < sA \leq B \leq tA$ for some constants $s, t$. Then, for all $\alpha \in [0 , 1]$
\begin{align}\label{e6}
(1-\alpha) A + \alpha B \leq  \exp \big( \frac{1}{2} \alpha ( 1- \alpha) \big( \frac{\max \lbrace 1, t \rbrace}{\min \lbrace 1, s \rbrace}-1\big)^2 \big)  A \sharp_\alpha B. 
\end{align}
\end{lemma}
By using this new ratio we can express some other operator reverse inequalities. The proofs are similar to that of preceding section. 
\begin{proposition}\label{p2}
 Let $g $ be a non-negative operator monotone decreasing function on $(0, \infty)$ and $0 < sA \leq B \leq tA$ for some constants $s, t$. Then, for all $\alpha \in [0 , 1]$
\begin{align*} \label{}
g (A \sharp_\alpha B)  \leq \exp \big( \frac{1}{2} \alpha ( 1- \alpha) \big( \frac{\max \lbrace 1, t \rbrace}{\min \lbrace 1, s \rbrace}-1\big)^2 \big)  g(A)\sharp_\alpha g(B).
\end{align*}
\end{proposition}
\begin{proof}
The assertion is obtained similar to the proof of Proposition \ref{p1}, by applying inequality \eqref{e6} instead of inequality \eqref{e1}. Note that for every $0\leq \alpha \leq 1$ and $s,t>0$,
$\exp \big( \frac{1}{2} \alpha ( 1- \alpha) \big( \frac{\max \lbrace 1, t \rbrace}{\min \lbrace 1, s \rbrace}-1\big)^2 \big)\geq1$.
\end{proof}
\begin{theorem}
 Let $ g $ be a non-negative operator monotone decreasing function on $(0, \infty)$ , $\frac{1}{p}+ \frac{1}{q}=1, p,q>1$,  and $0 < s A^p \leq B^q  \leq t A^p$ for some constants $s, t$. Then, for all $\xi \in \mathcal{H}$
\begin{align*} 
&g (A^p \sharp_\frac{1}{q} B^q)  \leq \exp \big( \frac{1}{2 pq} \big( \frac{\max \lbrace 1, t \rbrace}{\min \lbrace 1, s \rbrace}-1\big)^2 \big)g(A^p)\sharp_\frac{1}{q} g(B^q),\\ 
& \langle g(A^p \sharp_\frac{1}{q}B^q) \xi , \xi \rangle \leq \exp \big( \frac{1}{2 pq} \big( \frac{\max \lbrace 1, t \rbrace}{\min \lbrace 1, s \rbrace}-1\big)^2\big) \langle g(A^p) \xi , \xi \rangle ^\frac{1}{p} \langle g(B^q) \xi , \xi \rangle^\frac{1}{q}.
\end{align*}
\end{theorem}
%\begin{corollary}
%Let $\frac{1}{p}+\frac{1}{q}=1$, $p, q\geq 1$, and $A, B$ be commuting positive operators with spectra contained in $(0, 1)$ such that $0 < s A^p \leq  B^q  \leq t A^p$ for some constants $s, t$. Then, for every unit vector $\xi \in \mathcal{H}$
%\begin{align*}
%1- \| (A B)^\frac{1}{2} \xi \|^2 \leq \exp \big( \frac{1}{2 pq} \big( \frac{\max \lbrace 1, t \rbrace}{\min \lbrace 1, s \rbrace}-1\big)^2 \big) (1-\| A^\frac{p}{2} \xi \|^2)^\frac{1}{p} (1-\| B^\frac{q}{2} \xi \|^2)^\frac{1}{q}.
%\end{align*}
%\end{corollary}
%\begin{corollary}
% Let $ g $ be a non-negative operator monotone decreasing function on $(0, \infty)$ and $A, B$ be commuting positive operators such that $0 < s A^p \leq B^q  \leq t A^p$ for some constants $s, t$.Then
%\begin{align*}
%g(AB) \leq \exp \big( \frac{1}{2 pq} \big( \frac{\max \lbrace 1, t \rbrace}{\min \lbrace 1, s \rbrace}-1\big)^2 \big) g(A^P)^\frac{1}{p} g(B^q)^\frac{1}{q}.
%\end{align*}
%\end{corollary}

%\begin{corollary}
% Let $ g $ be a non-negative decreasing function on $(0, \infty)$ and $a_i, b_i$ be positive numbers such taht $ 0 < s \leq \frac{(b_i)^q}{(a_i)^p } \leq t $ for some constants $s, t$. Then
%\begin{align*}
%\sum_{i=1}^n g(a_i b_i) \leq \exp \big( \frac{1}{2 pq} \big( \frac{\max \lbrace 1, t \rbrace}{\min \lbrace 1, s %\rbrace}-1\big)^2 \big) \big( \sum_{i=1}^n g(a_i ^p) \big)^\frac{1}{p}  \big( \sum_{i=1}^n g(b_i ^q) \big)^\frac{1}{q}. 
%\end{align*}
%\end{corollary}
In \cite[Theorem B]{FM} another reverse Young's inequality is presented as follows: 
\begin{lemma}\label{l5}
Let $A$ and $B$ be positive operators such that $0 <s A \leq B \leq A$ for a constant $s$ and $\alpha \in [0 , 1]$. Then 
\begin{align*}
(1-\alpha) A + \alpha B \leq  M_\alpha (s) (A \sharp_\alpha B),
\end{align*}
where $M_\alpha(s) = 1+ \dfrac{\alpha (1-\alpha) (s-1)^2}{2s^{\alpha+1}}$.
\end{lemma}
Now by using this new constant, the similar reverse Acz\'{e}l inequalities are obtained. Note that $M_\alpha(s) \geq 1$ for every $\alpha \in [0 , 1]$. See \cite{FM} for more properties of $M_\alpha(s)$.

\begin{proposition}\label{p3}
 Let $g$ be a non-negative operator monotone decreasing function on $(0, \infty)$, $0 < sA \leq B \leq A$ for a constant $s$ and $\alpha \in [0 , 1]$. Then
\begin{align*} \label{}
g (A \sharp_\alpha B)  \leq  M_\alpha (s) (g(A)\sharp_\alpha g(B)).
\end{align*}
\end{proposition}
\begin{theorem}
 Let $ g $ be a non-negative operator monotone decreasing function on $(0, \infty)$ , $\frac{1}{p}+ \frac{1}{q}=1, p,q>1$,  and $0 < s A^p \leq B^q  \leq A^p$ for a constant $s$. Then, for all $\xi \in \mathcal{H}$
\begin{align*} 
&g (A^p \sharp_\frac{1}{q} B^q)  \leq M_{\frac{1}{q}}(s) (g(A^p)\sharp_\frac{1}{q} g(B^q)),\\ 
& \langle g(A^p \sharp_\frac{1}{q}B^q) \xi , \xi \rangle \leq M_{\frac{1}{q}} (s) \langle g(A^p) \xi , \xi \rangle ^\frac{1}{p} \langle g(B^q) \xi , \xi \rangle^\frac{1}{q}.
\end{align*}
\end{theorem}
\begin{remark}
We clearly see that the condition $0 < sA \leq B \leq tA $ for some $s\leq t$ in Lemma \ref{l3} is more general than the condition $0 < sA \leq B \leq A$ for $s\leq 1$ in Lemma \ref{l5}. But under the same condition $0<sA \leq B \leq A$, the appeared constant in Lemma \ref{l5} gives a better estimate than ones in Lemma \ref{l3}. In fact, we have 
\begin{align*} \label{}
M_\alpha (s) \leq \exp \big( \frac{1}{2} \alpha ( 1- \alpha) \big( \frac{1}{s}-1\big)^2 \big),
\end{align*}
for every $\alpha \in [0, 1]$ and $0<s \leq 1$ \cite[Proposition 2.9]{F1}.
\end{remark}

In \cite{GK} it is shown that if $f :  [0,\infty) \longrightarrow [0,\infty)$ is an operator monotone function and $0 <s A \leq B \leq t A$ for some constants $s, t$, then for all $\alpha \in [0 , 1]$
\begin{align*} \label{}
f(A)\sharp_\alpha f(B)  \leq max \lbrace S(s), S(t) \rbrace f( A \sharp_\alpha B ), 
\end{align*}
where $S(t)=\dfrac{t^{\frac{1}{t-1}}}{e \log t^{\frac{1}{t-1}}}$, for $t>0$ is the so called Specht's ratio.
As a result, we can show for a non-negative operator monotone decreasing function $g$ on $(0, \infty)$, $0 <s A \leq B \leq t A$, and $\alpha \in [0 , 1]$
\begin{align} \label{e17}
g (A \sharp_\alpha B)  \leq max \lbrace S(s), S(t) \rbrace  ( g(A)\sharp_\alpha g(B)).
\end{align}
Hence, one can deduce another reverse of operator Acz\'{e}l inequlity with the constant $max \lbrace S(p), S(q) \rbrace$, which is independent to $\alpha$.

\begin{remark}
In this paper, three evaluation expressions are derived. In the following, we show that there is no ordering between the appeared estimates.
\begin{enumerate}
\item[(1)] Comparison of the constants in Lemma \ref{l1} and in Lemma \ref{l3}: \\
 Let $0 < s A \leq B  \leq t A$ for some constants $s, t$ and $\alpha \in [0, 1]$. Also, with no loss of generality let  $s< t < 1$. Since $K$ is decreasing function on $(0, 1)$, by Lemma \ref{l1} we have 
 \begin{align*} 
(1-\alpha ) A+ \alpha B \leq  K(s)^R ( A \sharp_\alpha B ),
\end{align*}
where $R=\max \lbrace \alpha, 1-\alpha \rbrace$. Also, be Lemma \ref{l3}
\begin{align*}
(1-\alpha) A + \alpha B \leq  \exp \big( \frac{1}{2 }\alpha (1-\alpha) \big( \frac{1}{ s }-1\big)^2 \big)  A \sharp_\alpha B. 
\end{align*}
Now, the following numerical examples show that there is no ordering between them:
\begin{itemize}
\item[(i)] Take $\alpha=0.9$ and $s=0.3$, then we have
\begin{align*}
\max \lbrace  K(s)^\alpha, K(s)^{1-\alpha} \rbrace - \exp \big( \frac{1}{2 }\alpha (1-\alpha) \big( \frac{1}{ s }-1\big)^2 \big) \simeq 0.0833059.
\end{align*}

\item[(ii)] Take $\alpha=0.3$ and $s=0.3$, then we have
\begin{align*}
\max \lbrace  K(s)^\alpha, K(s)^{1-\alpha} \rbrace - \exp \big( \frac{1}{2 }\alpha (1-\alpha) \big( \frac{1}{ s }-1\big)^2 \big) \simeq -0.500368.
\end{align*}
\end{itemize}

\item[(2)] Comparison of the constants in the inequality \eqref{e17}  and in the inequality \eqref{e18} of Proposition \ref{p1}: \\
 Let $0 < s A \leq B  \leq t A$ for some constants $s, t$ and $\alpha \in [0, 1]$. Also, with no loss of generality let $1\leq s \leq t$. Then, from inequality \eqref{e17} we have 
\begin{align*} \label{}
g (A \sharp_\alpha B)  \leq S(t)  ( g(A)\sharp_\alpha g(B)),
\end{align*}
and from inequality \eqref{e18}
\begin{align*} \label{}
 g( A \sharp_\alpha B) \leq  K(t)^R  (g(A)\sharp_\alpha g(B)),
\end{align*}
where $R= \max \lbrace \alpha, 1-\alpha \rbrace$. We compare coefficients of these inequalitis as follows:
\begin{itemize}
\item[(i)] Take $\alpha=0.8$ and $t=9$, then 
\begin{align*}
\max \lbrace  K(t)^\alpha, K(t)^{1-\alpha} \rbrace - S(t) \simeq 0.501632.
\end{align*}

\item[(ii)] Take $\alpha=0.1$ and $t=9$, then 
\begin{align*}
\max \lbrace  K(t)^\alpha, K(t)^{1-\alpha} \rbrace - S(t) \simeq -0.655227.
\end{align*}
\end{itemize}
\end{enumerate}
\end{remark}

%{\bf Acknowledgement.} The authors are gratefull to Dr. Minghua Lin for the paper \cite{L4} and his several comments on a draft of this paper. 
\bibliographystyle{amsplain}

\end{document}